\documentclass[preprint,12pt]{elsarticle}
\usepackage{amsmath}
\usepackage{amssymb}

\input{tcilatex}

\begin{document}

\begin{frontmatter}



\title{
The core of the games with fractional linear utility functions
}

\author{Monica Patriche}

\address{
University of Bucharest
Faculty of Mathematics and Computer Science
    
14 Academiei Street
   
 010014 Bucharest, 
Romania
    
monica.patriche@yahoo.com }

\begin{abstract}
    We consider fractional linear programming production games for the single-objective
 and multiobjective cases. We use the method of Chakraborty and Gupta (2002) in order
 to transform the fractional linear programming problems into linear programming problems. 
A cooperative game is attached and we prove the non-emptiness of the core by using the 
duality theory from the linear programming. In the multiobjective case, we give a characterization 
of the Stable outcome of the associate cooperative game, which is balanced. We also consider 
the cooperative game associated to an exchange economy with a finite number of agents.
    \end{abstract}


\end{frontmatter}



\label{}





\bibliographystyle{elsarticle-num}
\bibliography{<your-bib-database>}

\begin{thebibliography}{99}
\bibitem{bon} [3] O.N. Bondareva, "Some applications of linear programming
methods to the theory of cooperative games" (in Russian), Problemy
Kibernetiki l0 (1963) 119-139.

\bibitem{Cha} M. Chakraborty, Sandipan Gupta, Fuzzy mathematical programming
for multi objective linear fractional programming problem, Fuzzy Sets and
Systems 125 (2002) 335--342.

\bibitem{charnes} A. Charnes, W. W. Cooper, Programming with linear
fractional functionals, Nav. Res. Logistics Quart. 9 (1962) 181-186.

\bibitem{Curiel} I. Curiel, J. Derks and S. Tijs, On balanced games and
games with committee control, OR Spectrum 11 (1989), 83-88.

\bibitem{gelekom} J. R. G. van Gellekom, J. A. M. Potters, J. H. Reijnierse,
M. C. Engel and S. H. Tijs, Characterization of the Owen Set of Linear
Production Processes, Games and Economic Behavior 32, 139-156 (2000).

\bibitem{granot} D. Granot, A generalized linear production model: a
unifying model, Mathematical Programming 34 (1986) 212-222.

\bibitem{Nis 2} I. Nishizaki , M. Sakawa, The Core of Multiobjective Linear
Production Programming Games, Electronics and Communications in Japan, Part
3, Vol. 82, No. 5, 1999

\bibitem{Nis4} I. Nishizaki , M. Sakawa, On computational methods for
solutions of multiobjective linear production programming games, European
Journal of Operational Research 129 (2001) 386-413.

\bibitem{owen} Guillermo Owen, On the core of linear production games,
Mathematical Programming 9 (1975) 358- 370.

\bibitem{Samel} D. Samet and E. Zemel, On the core and dual set of linear
programming games, Mathematics and Operations Research, Vol. 9, No. 2, May
1984

\bibitem{Sch1} S. Schaible, Fractional programming I: duality, Manage. Sci.
A 22 (1976) 658-667.

\bibitem{Sch 2} S. Schaible, Analyse and Anwendungen von
Quotientenprogrammen, Verlag Anton Hain, Meisenheim am Glan, 1978.

\bibitem{shapley} [!6] L.S. Shapley, "On balanced sets and cores", Naval
Research Logistics Quarterly 14 (1967) 453-460.

\bibitem{zim} H.-J. Zimmermann, Fuzzy programming and linear programming
with several objective functions, Fuzzy Sets and Systems 1 (1978) 45-55.
\end{thebibliography}







\section{INTRODUCTION}

Owen considered the linear programming production problem with $n$ producers
who have $m$ resources and cooperate in order to produce $p$ goods. The
producers' aim is the maximization of their income, which is modeled as the
objective function of the discussed problem. A cooperative game is attached
and the fair allocation of the income is put into question. Methods from
duality theory are related to methods from cooperative games in order to
prove that the core is nonempty and to find its elements. The condition of
the nonemptiness of the core for a cooperative game is the balancedness, as
it was proved in [bon, shapley]

The seminal work of Owen has many extensions. Samet and Zemel (1984) studied
the relation between the core of a given LP-game and the set of payoff
vectors generated by optimal dual solutions to the corresponding linear
program. Granot (1984) generalized the Owen's model so that the resources
held by any subset of producers S is not restricted to be the vector sum of
the resources held by the members of S. He also proved the non-emptiness of
the core of the associated game. Curiel, Derks and Tijs (1989) considered
linear production games with committee control. Gellekom, Potters,
Reijnierse, Engel and Tijs (2001) also studied linear production processes.
Nishizaki and Sakawa (1999, 2001) treated the multiobjective case.

In this paper, we consider that the producers want to maximize the average
income on unit time, which is modeled by a fractional linear objective
function. We generalize the Owen's model by introducing the fractional
linear programming production games for the single-objective and
multiobjective cases. The transformation of the fractional linear
programming problems into linear programming problems is made by using the
method of Chakraborty and Gupta (2002). We attach a cooperative game and we
prove the non-emptiness of its core. The multiobjective game is balanced,
but not superaditive. For this case, we give a characterization of the
Stable outcome of the associate cooperative game. Finally, we consider the
cooperative game associated to an exchange economy with a finite number of
agents.

The paper is organized in the following way: fractional linear programming
production game is presented in Section 2, the multiobjective model is
studied in Section 3 and the fractional linear programming production game
with fuzzy parameters is the containt of Section 4.\bigskip

\section{Fractional linear programming production games}

\subsection{The model}

A colection $\mathcal{B}$ of coalitions is said to be balanced if there
exists $\gamma (S)>0$ for each $S\in \mathcal{B}$ such that, for each $i\in
N,$ $\tsum_{\substack{ S\in \mathcal{B}  \\ i\in S}}\gamma (S)=1.$ The
cooperative game $(N,V)$ is called balanced if for each balanced colection $%
\mathcal{B}$, $\tsum_{S\in \mathcal{B}}\gamma (S)V(S)\leq V(N).$

We consider here the following model of fractional linear production game.
There are $m$ types of resurces used for the production of $p$ goods. For
each $i\in N,$ the player $i$ is endowed with a vector $b^{i}$ of resources,
where $b^{i}=(b_{1}^{i},b_{2}^{i},...,b_{m}^{i})$. Any coalition $S$ will
use a total of $b_{k}(S)=\tsum_{i\in S}b_{k}^{i}$ units of the $k$th
resource. We assume that a unit of the $j$th good ($j=1,...,p$) requires $%
a_{kj}$ units of the $k$ th resource $(k=1,...,m)$. A coalition $S$ uses all
its resources in order to produce a vector $(x_{1},x_{2},...,x_{p})$ of
goods which satisfies\medskip

$a_{11}x_{1}+a_{12}x_{2}+...+a_{1p}x_{p}\leq b_{1}(S)$

$a_{21}x_{1}+a_{22}x_{2}+...+a_{2p}x_{p}\leq b_{2}(S)$

................................................ \ \ \ \ \ \ \ \ \ \ \ \ \ \
\ \ \ \ \ \ \ \ \ \ \ \ \ \ \ \ (1)

$a_{m1}x_{1}+a_{m2}x_{2}+...+a_{mp}x_{p}\leq b_{m}(S)$

$x_{1},x_{2},...,x_{p}\geq 0.\medskip $

The players of the coalition $S$ want to maximize the function $\frac{N(x)}{%
D(x)}=\frac{c_{1}x_{1}+c_{2}x_{2}+...+c_{p}x_{p}+c_{0}}{%
d_{1}x_{1}+d_{2}x_{2}+...+d_{p}x_{p}+d_{0}},$ which means the average income
on unit time as in Tigan [].

We will denote by $%
x=(x_{1},x_{2},...,x_{p})^{T},c=(c_{1},c_{2},...,c_{p})^{T},d=(d_{1},d_{2},...,d_{p})^{T}\in 
\mathbb{R}^{p},$ $A=(a_{ij})_{\substack{ i=\overline{1,m}  \\ j=\overline{1,p%
} }}\in \mathbb{R}^{m\times p},$ $b(S)=(b_{1}(S),b_{2}(S),...,b_{p}(S))^{T}%
\in \mathbb{R}^{m},$ $D_{S}=\{x\in R^{p}:Ax\leq b(S),$ $x\geq 0\}\subseteq 
\mathbb{R}^{p}.$ \bigskip

After we make the substitution (Charnes and Cooper) $y=tx,$ $t=\frac{1}{%
dx+d_{0}},$ the problem $(1)$ becomes equivalent with $(2)$, where\medskip

Max $tN(\frac{y}{t})$

$tD(\frac{y}{t})\leq 1$ \ \ \ \ \ \ \ \ \ \ \ \ \ \ \ \ \ \ \ \ \ \ \ \ \ \
\ \ \ \ \ \ \ \ \ \ \ \ \ \ \ \ \ \ \ \ (2)

$A(\frac{y}{t})-b(S)\leq 0$

$t>0,y\geq 0.\medskip $

Assume that $x\in \Delta :=\{x:Ax\leq b,$ $x\geq 0\}$ implies $D(x)>0.$

The problem $(2)$ can be reduced to problem $(3):$\medskip

$P(S)$\bigskip

Max $cy+c_{0}t$

$dy+d_{0}t\leq 1$

$Ay-tb(S)\leq 0$ \ \ \ \ \ \ \ \ \ \ \ \ \ \ \ \ \ \ \ \ \ \ \ \ \ \ \ \ \ \
\ \ \ \ \ \ \ \ \ \ \ \ \ \ \ \ \ \ \ (3)

$t>0,y\geq 0.$

We present here several results conserning the relation between the above
problems.

\begin{theorem}
(Schaible [,]). Let for some $\xi \in \Delta ,$ $N(\xi )\geq 0,$ if (1)
reaches a (global) maximum at $x=x^{\ast },$ then (3) reaches a (global)
maximum at a point $(t,y)=(t^{\ast },y^{\ast }),$ where $\frac{y^{\ast }}{%
t^{\ast }}=x^{\ast }$ and the objective functions at these points are
equal.\medskip
\end{theorem}

\begin{theorem}
(Schaible []) If (1) reaches a (global) maximum at a point $x^{\ast },$ then
the corresponding transformed problem (3) attains the same maximum value at
a point $(t^{\ast },y^{\ast })$ where $x^{\ast }=\frac{y^{\ast }}{t^{\ast }}%
. $ Moreover (3) has a concave objective function and a convex feasible set.
\end{theorem}

\subsection{The associated cooperative game}

We associate to the problem above the cooperative game $(N,V),$ where $N$ is
the set of players and $V:\mathcal{P}(N)\rightarrow \mathbb{R},$ where $V$
is defined by\medskip

$V(S)=c_{1}y_{1}+c_{2}y_{2}+...+c_{p}y_{p}+c_{0}t$ if $S\subsetneq N,$ where 
$y$ is an optimal solution to problem $P(S)$\medskip

and\medskip

$V(N)=\gamma (c_{1}y_{1}+c_{2}y_{2}+...+c_{p}y_{p}+c_{0}t),$ where $y$ is an
optimal solution to problem $P(N)$

and $\gamma >\max (\gamma ^{\ast },n)$ and $\gamma ^{\ast }=\max_{\mathcal{B}%
}\tsum_{S\in \mathcal{B}}\gamma (S),$ $\mathcal{B}$ being any balanced
coalition of $N.$\medskip

We prove first the nonemptiness of the core of the game $(N,V).$ This fact
is a consequence of fact that the game is balanced.

\begin{theorem}
The game $(N,V)$ is balanced.
\end{theorem}

\begin{proof}
Let $\mathcal{B}$ be a balanced colection of $N.$ First, we have that

$\tsum_{S\in \mathcal{B}}\gamma (S)b_{k}(S)=b_{k}(N)$ for each $k\in
\{1,2,...,m\},$ and then,

$\tsum_{S\in \mathcal{B}}\gamma (S)v(S)=$

$=\tsum_{S\in \mathcal{B}}\gamma
(S)(c_{1}y_{1}(S)+c_{2}y_{2}(S)+...+c_{p}y_{p}(S)+c_{0}t(S))=$

$=\tsum_{j=1}^{p}c_{j}(\tsum_{S\in \mathcal{B}}\gamma
(S)y_{j}(S))+c_{0}\tsum_{S\in \mathcal{B}}\gamma (S)t(S)=$

$=\gamma (\tsum_{j=1}^{p}c_{j}\widehat{y}_{j}+c_{0}\widehat{t}),$ where

$\gamma ^{^{\prime }}=\tsum_{S\in \mathcal{B}}V(S),$ $\widehat{y}%
_{j}:=\tsum_{S\in \mathcal{B}}\frac{\gamma (S)}{\gamma ^{\prime }}y_{j}(S)$
and $\widehat{t}=\tsum_{S\in \mathcal{B}}\frac{\gamma (S)}{\gamma ^{\prime }}%
t(S).$

Assume that $(y(S),t(S))\in D_{S}.$

Since $Ay(S)\leq t(S)b(S)$ for each $S\in \mathcal{B}$ and $\frac{\gamma (S)%
}{\gamma ^{\prime }}\geq 0,$ it follows that

$A\frac{\gamma (S)}{\gamma ^{\prime }}y(s)\leq t(S)\frac{\gamma (S)}{\gamma
^{\prime }}b(S)$. By adding, we obtain $A(\tsum_{S\in \mathcal{B}}\frac{%
\gamma (S)}{\gamma ^{\prime }}y(s))\leq (\tsum_{S\in \mathcal{B}}\frac{%
\gamma (S)}{\gamma ^{\prime }}b(S))t(S),$ and then $A(\tsum_{S\in \mathcal{B}%
}\frac{\gamma (S)}{\gamma ^{\prime }}y(s))\leq b(N)(\tsum_{S\in \mathcal{B}}%
\frac{\gamma (S)}{\gamma ^{\prime }}t(S)).$ Therefore, $A\widehat{y}\leq b(N)%
\widehat{t}$, that is $(\widehat{y},\widehat{t})$ verifies $Ay-tb(N)\leq 0$

Since $dy(S)+d_{0}t(S)\leq 1$ for each $S\in \mathcal{B}$ and $\frac{\gamma
(S)}{\gamma ^{\prime }}\geq 0,$ it follows that $d\tsum_{S\in \mathcal{B}}%
\frac{\gamma (S)}{\gamma ^{\prime }}y(S)+d_{0}\tsum_{S\in \mathcal{B}}\frac{%
\gamma (S)}{\gamma ^{\prime }}t(S)\leq 1$ and then, $d\widehat{y}+d_{0}%
\widehat{t}\leq 1.$

We notice that $\widehat{y}_{j},\widehat{t}\geq 0$ and conclude that $(%
\widehat{y},\widehat{t})$ is a feasible solution for the linear problem
associated to the coalition $N$ and $V(N)\geq \gamma ^{\prime }(c_{1}%
\widehat{y}_{1}+c_{2}\widehat{y}_{2}+...+c_{p}\widehat{y}_{p}+c_{0}\widehat{t%
}).$
\end{proof}

\begin{corollary}
The core of the game $(N,V)$ is nonempty.\medskip
\end{corollary}

Now, we find the elements of the core for the game $(N,V).$

We denote $c^{\prime }=(c,c_{0}),$ $y^{\prime }=(y,t)\in \mathbb{R}%
^{p^{+1}}, $ $b^{\prime }(S)=(d_{0},-b(S))\in R^{m+1}A^{\prime }(S)=%
\begin{pmatrix}
d & d_{0} \\ 
A & -b(S)%
\end{pmatrix}%
\in \mathbb{R}^{(m+1)\times (p+1)}.\medskip $

We write the primal problem (3) as\medskip

$P(S):$ \ \ \ \ \ \ \ Max$c^{\prime }y^{\prime }$

$\ \ \ \ \ \ \ \ \ \ \ \ \ \ \ \ A^{\prime }(S)y^{\prime }\leq
(1,0,..,0,0)^{T}$

$\ \ \ \ \ \ \ \ \ \ \ \ \ \ \ \ y^{\prime }\geq 0.$

\bigskip

The dual of $P(S)$ is the problem $D(S),$ where\medskip

$D(S):$ \ \ \ \ \ \ \ Min$\omega _{1}$

$\ \ \ \ \ \ \ \ \ \ \ \ \ \ \ \ (A^{\prime }(S))^{T}\omega \geq c^{\prime }$

$\ \ \ \ \ \ \ \ \ \ \ \ \ \ \ \ \omega _{i}\geq 0$ for each $%
i=1,...,m+1.\medskip $

Explicitely, for each $S\subsetneq N,$ $D(S)$ is\medskip

$D(S):$ \ \ \ \ \ \ \ Min$\omega _{1}$

$\ \ \ \ \ \ \ \ \ \ \ \ \ \ \ \ d\omega _{1}+A^{T}(\omega _{2},\omega
_{3},...,\omega _{m+1})^{T}\geq c$

\ \ \ \ \ \ \ \ \ \ \ \ \ \ \ \ $d_{0}\omega _{1}-b(S)(\omega _{2},\omega
_{3},...,\omega _{m+1})^{T}\geq c_{0}$

$\ \ \ \ \ \ \ \ \ \ \ \ \ \ \ \ \omega _{i}\geq 0$ for each $%
i=1,...,m+1.\medskip $

and $D(N)$ is Min$\omega _{1}$

$\ \ \ \ \ \ \ \ \ \ \ \ \ \ \ \ d\omega _{1}+A^{T}(\omega _{2},\omega
_{3},...,\omega _{m+1})^{T}\geq \gamma c$

\ \ \ \ \ \ \ \ \ \ \ \ \ \ \ \ $d_{0}\omega _{1}-b(S)(\omega _{2},\omega
_{3},...,\omega _{m+1})^{T}\geq \gamma c_{0}$

$\ \ \ \ \ \ \ \ \ \ \ \ \ \ \ \ \omega _{i}\geq 0$ for each $%
i=1,...,m+1.\medskip $

\bigskip Let $\omega ^{\ast }=(\omega _{1}^{\ast },\omega _{2}^{\ast
},...,\omega _{m+1}^{\ast })^{T}$ be a solution of $D(N).$ Then, $%
V(N)=\omega _{1}^{\ast }$ and let $u=(u_{1},u_{2},...,u_{n})^{T}$ such that $%
u_{i}=\frac{1}{n}\omega _{1}^{\ast }$ for each $i\in \{1,2,...,n\}.$

We will prove that $(u_{1},u_{2},...,u_{n})^{T}\in $core$(N,V).$

First, $\tsum_{i\in N}u_{i}=v(N).$

We must show, in addition, that $\tsum_{i\in S}u_{i}\geq v(S)$ for each $%
S\subset N.$

Notice that $\tsum_{i\in S}u_{i}=\frac{|S|}{n}\omega _{1}^{\ast }.$

The vector $(\frac{\omega _{1}^{\ast }}{\gamma },\frac{\omega _{2}^{\ast }}{%
\gamma },...,\frac{\omega _{m+1}^{\ast }}{\gamma })$ verifies the
restrictions of $D(S):$

$d_{0}\frac{\omega _{1}^{\ast }}{\gamma }=\frac{b(S)}{\gamma }(\omega
_{2}^{\ast },...,\omega _{m+1}^{\ast })^{T}=d_{0}\frac{\omega _{1}^{\ast }}{%
\gamma }-\frac{b(N)}{\gamma }(\omega _{2}^{\ast },...,\omega _{m+1}^{\ast
})^{T}+\frac{b(N)-b(S)}{\gamma }(\omega _{2}^{\ast },...,\omega _{m+1}^{\ast
})^{T}\geq c_{0}$

and

$d\frac{\omega _{1}^{\ast }}{\gamma }=\frac{1}{\gamma }A^{T}(\omega
_{2}^{\ast },...,\omega _{m+1}^{\ast })^{T}\geq c.$

It follows that $V(S)\leq \frac{\omega _{1}^{\ast }}{\gamma }.$

Since $\gamma \geq n,$ we have that $\frac{1}{\gamma }\leq \frac{1}{n}\leq 
\frac{|S|}{n}$ for each $S\subseteq N$ and then, $\tsum_{i\in S}u_{i}=\frac{%
|S|}{n}\omega _{1}^{\ast }\geq \frac{1}{\gamma }\omega _{1}^{\ast }\geq V(S)$
for each $S\subseteq N.$

We conclude that $(u_{1},u_{2},...,u_{n})^{T}\in $core$(N,V).$

\section{Multiobjective fractional linear programming production games}

\subsection{The model}

We prove that the game is balanced, but it is not superaditive. We give a
characterization of the Stable outcome of the associate cooperative
game.\medskip

Let $n$ be a fixed positive integer, let $N=\{1,2,...,n\}$ be the set of
players and $\mathcal{P}(N)$, the set of nonempty subsets of $N$ being the
set of coalitions formed with players $1,2,...,n.$

For each coalition $S\subseteq N,$ define the problem $P(S):$

\bigskip

$P(S):$

max $z_{1}(x)=\frac{N_{1}(x)}{D_{1}(x)}=\frac{c_{1}x+c_{10}}{d_{1}x+d_{10}}$

\ \ \ \ \ \ $z_{2}(x)=\frac{N_{2}(x)}{D_{2}(x)}=\frac{c_{2}x+c_{20}}{%
d_{2}x+d_{20}}$

...

\ \ \ \ \ $z_{r}(x)=\frac{N_{r}(x)}{D_{r}(x)}=\frac{c_{r}x+c_{r0}}{%
d_{r}x+d_{r0}}$

$a_{11}x_{1}+a_{12}x_{2}+...+a_{1p}x_{p}\leq b_{1}(S)$

$a_{21}x_{1}+a_{22}x_{2}+...+a_{2p}x_{p}\leq b_{2}(S)$

................................................ \ \ \ \ \ \ \ \ \ \ \ \ \ \
\ \ \ \ \ \ \ \ \ \ \ \ \ \ \ \ (4)

$a_{m1}x_{1}+a_{m2}x_{2}+...+a_{mp}x_{p}\leq b_{m}(S)$

$x_{1},x_{2},...,x_{p}\geq 0,\medskip $

where $c_{i},d_{i},x\in \mathbb{R}^{p}$ for each $i\in \{1,2,...,r),$ $%
A=(a_{ij})_{\substack{ i=\overline{1,m}  \\ j=\overline{1,p}}}\in \mathbb{R}%
^{m\times p},$ $b(S)=(b_{1}(S),b_{2}(S),...,b_{p}(S))^{T}\in \mathbb{R}^{m},$
$D_{S}=\{x\in \mathbb{R}^{p}:Ax\leq b(S),$ $x\geq 0\}\subseteq \mathbb{R}%
^{p}.\medskip $

Let $F(x)=(\frac{c_{1}x+c_{10}}{d_{1}x+d_{10}},...,\frac{c_{r}x+c_{r0}}{%
d_{r}x+d_{r0}}).$

Let $I=\{i:N_{i}(x)\geq 0$ for some $x\in D\}$

\ \ \ $I^{C}=\{i:N_{i}(x)\leq 0$ for every $x\in D\}.$

Suppose $D$ is nonempty and bounded.

We present here a model developement due to Chakraborty and Gupta (2002).

Let $t=\cap _{i\in I}\frac{1}{d_{i}x+d_{i0}}$ and $t=\cap _{i\in I^{C}}\frac{%
-1}{c_{i}x+c_{i0}}\Leftrightarrow \frac{1}{d_{i}x+d_{i0}}\geq t$ for each $%
i\in I$ and $\frac{-1}{c_{i}x+c_{i0}}\geq t$ for each $i\in I^{C}$ and $%
y=tx. $

The multiobjective linear fractional programming problem (4) is equivalent
with the multiobjective linear programming problem (5):\medskip

Max $\{tN_{i}(\frac{y}{t})$ if $i\in I,$ $tD_{i}(\frac{y}{t})$ $i\in I^{C}\}$

$tD_{i}(\frac{y}{t})\leq 1$ if $i\in I$

$tN_{i}(\frac{y}{t})\leq 1$ if $i\in I^{C}$\ \ \ \ \ \ \ \ \ \ \ \ \ \ \ \ \
\ \ \ \ \ \ \ \ \ \ \ \ \ \ \ \ \ \ \ \ \ \ \ \ \ \ \ \ \ (5)

$A(\frac{y}{t})-b(S)\leq 0$

$t,y\geq 0.\medskip $

Chakraborty and Gupta (2002) proved that the constraint set of (5) is
non-empty convex set having feasible points.

We will assume further that $I^{C}=\emptyset .$

We will use the following notations:

$T_{S}=\{(t,y):(t,y)$ verifies the restriction of the problem $(2)\};$

$\widehat{T}_{S}=\{(z\in \mathbb{R}^{r}:z=F(t,y),$ $(t,y)\in T_{s}\}$

$V(S)=($Max$\widehat{T}_{S}-\mathbb{R}_{+}^{r})\cap \mathbb{R}_{+}^{r}.$

\subsection{The associated cooperative game}

We prove that the game is balanced, but it is not superaditive. We give a
characterization of the Stable outcome of the associate cooperative
game.\medskip

Let $N=\{1,2,...,n\}$ and $V:P(N)\rightarrow \mathbb{R}^{r}$ be defined by

$V(S)=($Max$\widehat{T}_{S}-\mathbb{R}_{+}^{r})\cap \mathbb{R}_{+}^{r}$ for
each $S\subset N$

and

$V(N)=($Max$\gamma \widehat{T}_{N}-\mathbb{R}_{+}^{r})\cap \mathbb{R}%
_{+}^{r} $

where $\gamma >\max (\gamma ^{\ast },n)$ and $\gamma ^{\ast }=\max_{\mathcal{%
B}}\tsum_{S\in \mathcal{B}}\gamma (S),$ $\mathcal{B}$ being any balanced
coalition of $N.$

The set of imputation of the game is

$I(N,V)=\{x\in \mathbb{R}_{+}^{m\times n}:x_{N}\in $Max$V_{N},$ $x_{i}\notin
V_{\{i\}}\backslash $Max$V_{\{i\}},$ $\forall i\in N\}$

The stable outcome is

$SO(N,V)=\{x\in \mathbb{R}_{+}^{m\times n}:x_{S}\notin V(S)\backslash $Max$%
V(S),$ $\forall S\subset N\}.$

\medskip The problem (5) is equivalent with the problem (6)

Max $c_{i}y+c_{i0}t,$ $i=1,...,m;$

$d_{i}y+d_{i0}t\leq 1$ $i=1,...,m;$

$A(y)-tIb(S)\leq 0$ \ \ \ \ \ \ \ \ \ \ \ \ \ \ \ \ \ \ \ \ \ \ \ \ \ \ \ \
\ \ \ \ \ \ \ \ \ \ \ \ \ \ \ \ \ \ \ \ \ (6)

$t,y\geq 0.\medskip $

\begin{theorem}
The game $(N,V)$ is balanced.\medskip
\end{theorem}

\begin{proof}
Let $\mathcal{B}$ be a balanced colection of $N.$ First, we have that

$\tsum_{S\subset N}\gamma (S)b_{k}(S)=\tsum_{S\subset N}\tsum_{i\in S}\gamma
(S)b_{k}^{i}(S)=\tsum_{i\in N}\{\tsum_{S\subset N,S\backepsilon i}\gamma
(S)\}b_{k}^{i}=\tsum_{i\in N}b_{k}^{i}=b_{k}(N)$

for each $k\in \{1,2,...,m\}.$

Let $z(y(S),t(S))=(z_{1}(y(S),t(S)),z_{2}(y(S),t(S)),...,z_{l}(y(S),t(S)))%
\in V(S)$ and then,

$\tsum_{S\in \mathcal{B}}\frac{\gamma (S)}{\gamma ^{\prime }}z(y(S),t(S))=$

$=\tsum_{S\in \mathcal{B}}\frac{\gamma (S)}{\gamma ^{\prime }}%
(...,c_{i}y(S)+y(S)c_{i0},...)=$

$=(\widehat{t}N_{1}(\frac{\widehat{y}}{\widehat{t}}),\widehat{t}N_{2}(\frac{%
\widehat{y}}{\widehat{t}}),...,\widehat{t}N_{l}(\frac{\widehat{y}}{\widehat{t%
}}))=$

$=z(\widehat{y},\widehat{t}),$

where $\gamma ^{^{\prime }}=\tsum_{S\in \mathcal{B}}\gamma (S),$ $\widehat{y}%
:=\tsum_{S\in \mathcal{B}}\frac{\gamma (S)}{\gamma ^{\prime }}y(S)$ and $%
\widehat{t}=\tsum_{S\in \mathcal{B}}\frac{\gamma (S)}{\gamma ^{\prime }}%
t(S), $ $\widehat{y}\in \mathbb{R}_{+}^{p},\widehat{t}\in \mathbb{R}_{+}.$

Assume that $(y(S),t(S))\in D_{S}.$

Since $Ay(S)\leq t(S)b(S)$ for each $S\in \mathcal{B}$ and $\frac{\gamma (S)%
}{\gamma ^{\prime }}\geq 0,$ it follows that

$A\frac{\gamma (S)}{\gamma ^{\prime }}y(s)\leq t(S)\frac{\gamma (S)}{\gamma
^{\prime }}b(S)$. By adding, we obtain $A(\tsum_{S\in \mathcal{B}}\frac{%
\gamma (S)}{\gamma ^{\prime }}y(s))\leq (\tsum_{S\in \mathcal{B}}\frac{%
\gamma (S)}{\gamma ^{\prime }}b(S))t(S),$ and then $A(\tsum_{S\in \mathcal{B}%
}\frac{\gamma (S)}{\gamma ^{\prime }}y(s))\leq b(N)(\tsum_{S\in \mathcal{B}}%
\frac{\gamma (S)}{\gamma ^{\prime }}t(S)).$ Therefore, $A\widehat{y}\leq b(N)%
\widehat{t}$, that is $(\widehat{y},\widehat{t})$ verifies $Ay-tb(N)\leq 0$

Since $dy(S)+d_{0}t(S)\leq 1$ for each $S\in \mathcal{B}$ and $\frac{\gamma
(S)}{\gamma ^{\prime }}\geq 0,$ it follows that $d\tsum_{S\in \mathcal{B}}%
\frac{\gamma (S)}{\gamma ^{\prime }}y(S)+d_{0}\tsum_{S\in \mathcal{B}}\frac{%
\gamma (S)}{\gamma ^{\prime }}t(S)\leq 1$ and then, $d\widehat{y}+d_{0}%
\widehat{t}\leq 1.$

We notice that $\widehat{y},\widehat{t}\geq 0$ and conclude that $(\widehat{y%
},\widehat{t})\in V(N)$ and therefore $\tsum_{S\subset N}\gamma
(S)(y(S),t(S))\in V(N)$....

Since $\tsum_{S\subset N}\gamma (S)V(S)\subset V(N),$ it follows that is the
game $(N,V)$ is balanced.
\end{proof}

We consider the dual problem to the multiobjective linear problem in order
to find a point beloging to the core. We present here some useful results
concerning the duality of the multiobjective linear programming.

Let the primal and the dual problems as follows:

max $z(x)=Cx$

\ \ \ \ \ $z\in T_{p}=\{x:Ax=b,$ $x\in \mathbb{R}_{+}^{p}\}$ \ \ \ \ \ \ \ \
\ \ \ \ \ \ \ \ \ \ \ \ \ (7)

and respectively

\ min $g(w)=wb$

\ \ \ \ \ \ \ $w\in T_{d}=\{w:wAu\leq Cu$ for no $u\in \mathbb{R}_{+}^{p}\},$
\ \ \ \ \ \ \ \ \ \ \ \ \ \ \ \ \ \ (8)

where $z(x)=(z_{1}(x),z_{2}(x),...,z_{r}(x)\},$ $%
g(w)=(g_{1}(w),g_{2}(w),...,g_{r}(w)),$ $C\in \mathbb{R}^{r+p},$ $A\in 
\mathbb{R}^{m+p},$ $b\in C\in \mathbb{R}^{m}.$

We will use the following theorems.

\begin{theorem}
If $x$ is a feasible solution of primal problem (7) and $w$ is a feasible
solution of dual problem (8), it is not the case that $g(w)\leq z(x).$
\end{theorem}

\begin{theorem}
Assume that $x^{\ast }$ is a feasible solution of primal problem (7) and $w$
is a feasible solution of dual problem (8). Also assume that $z(x^{\ast
})=g(w^{\ast })$ is satisfied. Then, $x^{\ast }$ is a Pareto optimal
solution of primal problem (7), and $w^{\ast }$ is a Pareto optimal solution
of dual problem (8).
\end{theorem}

\begin{theorem}
Considering main problem (7) and dual problem (8), the following two
statements are equivalent.
\end{theorem}

\ \ \ \ \ \ \ \ \ \ \ \ \ \ \ \textit{(1) Each of the problems has a
feasible solution.}

\textit{\ \ \ \ \ \ \ \ \ \ \ \ \ \ \ \ (2) Each of the problem has a Pareto
optimal solution, and there exists at least a pair of Pareto optimal
solutions such that }$z(x^{\ast })=g(w^{\ast }).$

\begin{theorem}
The necessary and sufficient condition for $x^{\ast }$ to be a Pareto
optimal solution of primal problem (7) is that there exists a feasible
solution $w^{\ast }$ of dual problem (8) such that $z(x^{\ast })=g(w^{\ast
}).$ Then, $w^{\ast }$ itself is a Pareto optimal solution of dual problem
(8).\medskip
\end{theorem}

The next theorem gives an element of the Stable outcome.

\begin{theorem}
Let $\omega ^{\ast }$ be a Pareto optimal solution of the dual problem of
the associated multiobjective linear programming problem $(3)$ with $S=N.$
Then the payoff $u=(u_{1\cdot },u_{2\cdot },...,u_{n\cdot })\in \mathbb{R}%
^{r\times n},$ $u_{i\cdot }=(u_{i1},u_{i2},...,u_{ir})$ defined by $u_{ik}=%
\frac{1}{n}\omega _{k1}^{\ast },$ $i=1,2,...,n$ and $k=1,2,...,r$ belongs to
the stable outcome of the game $(N,V).$
\end{theorem}

\begin{proof}
We will formulate first the multiobjective linear production programming
problems $P(S)$ for each $S\subseteq N$ equivalent to the multiobjective
fractional linear production programming problems and the dual.

Let $c^{\prime }\in \mathbb{R}^{r\times (p+1)},$ $A^{\prime }\in \mathbb{R}%
^{(r+m)\times (p+1)}$ and $b^{\prime }\in \mathbb{R}^{r+m}$ be defined as $%
c^{\prime }=(c_{1},e_{1},...,c_{r},c_{r0}),$ $A^{\prime }(S)=%
\begin{pmatrix}
(d_{i})_{i\in I} & (d_{i0})_{i\in I} \\ 
A & -I_{m}b(S)%
\end{pmatrix}%
b^{\prime }=(1_{I},0_{\mathbb{R}^{m}}),$ where $1_{I}=(1,1,...,1)\in \mathbb{%
R}^{l}$ and $0_{\mathbb{R}^{m}}=(0,0,...,0)\in \mathbb{R}^{m}.$
\end{proof}

\ \ \ \ \ \ \ \ \ \ \ \ \ \ \ \ \ \ \ \ \ \ For each $S\subsetneq N,$ the
problem $P(S)$ is\medskip

$P(S):$ \ \ \ \ \ \ \ Max $c^{\prime }(y,t)$

$\ \ \ \ \ \ \ \ \ \ \ \ \ \ \ \ (y,t)\in T_{p}=\{(y,t):A^{\prime
}(S)(y,t)=b^{\prime },$ $(y,t)\in \mathbb{R}_{+}^{p+1}\}\medskip $

and $P(N):$ \ \ \ \ \ \ \ Max $\gamma c^{\prime }(y,t)$

$\ \ \ \ \ \ \ \ \ \ \ \ \ \ \ \ (y,t)\in T_{p}=\{(y,t):A^{\prime
}(S)(y,t)=b^{\prime },$ $(y,t)\in \mathbb{R}_{+}^{p+1}\}\medskip $

Let $\widehat{T}_{S}$ , $\widehat{T}_{N}$ be the feasible areas in the
objective space of primal problems $P(S),$ resp. $P(N).$\medskip

\ \ \ \ \ \ \ \ \ \ \ \ \ \ \ \ \ \ \ \ \ For each $S\subsetneq N,$ the dual 
$D(S)$ is\medskip

$D(S):$ \ \ \ \ \ \ \ Min $\omega b^{\prime }$

$\ \ \ \ \ \ \ \ \ \ \ \ \ \ \ \ \omega \in T_{d}=\{\omega :\omega A^{\prime
}(S)u\leq c^{\prime }u$ for no $u\in \mathbb{R}_{+}^{p+1}\}\medskip $

and $D(N)$ is Min $\omega b^{\prime }$

$\ \ \ \ \ \ \ \ \ \ \ \ \ \ \ \ \omega \in T_{d}=\{\omega :\omega A^{\prime
}(N)u\leq \gamma c^{\prime }u$ for no $u\in \mathbb{R}_{+}^{p+1}\}\medskip $

Let $\omega ^{\ast }$ and $(y^{\ast },t^{\ast })$ be Pareto optimal
solutions for the problems $D(N)$ and $P(N)$.

It follows that $\omega ^{\ast }b^{\prime }=c^{\prime }(y^{\ast },t^{\ast })$
and then, $\omega ^{\ast }b^{\prime }\in $Max$\widehat{T}_{N}.$ We have that 
$\tsum_{i\in N}u_{i,\cdot }=\tsum_{i\in N}\frac{1}{n}\omega _{1,\cdot
}=(\omega _{11}^{\ast },...,\omega _{1r}^{\ast })\in $Max$V(N).$

For each $S\subsetneq N,$ $\tsum_{i\in S}u_{i,\cdot }=\tsum_{i\in N}\frac{|S|%
}{n}\omega _{1,\cdot }^{\ast }$

$\omega ^{\ast }A^{\prime }(N)u\leq \gamma c^{\prime }u$ for no $u\in 
\mathbb{R}_{+}^{p+1}$ implies that $\frac{\omega ^{\ast }}{\gamma }A^{\prime
}(S)u\leq c^{\prime }u$ for no $u\in \mathbb{R}_{+}^{p+1}.$ It follows that $%
\frac{\omega ^{\ast }b^{\prime }}{\gamma }\in V(S).$ Since $\gamma \geq n,$ $%
\frac{1}{\gamma }\leq \frac{1}{n}\leq \frac{|S|}{n}$ for each $S\subseteq N.$

$\tsum_{i\in S}u_{i,\cdot }=\tsum_{i\in N}\frac{|S|}{n}\omega _{1,\cdot
}^{\ast }\geq \frac{1}{\gamma }\omega _{1}^{\ast }.$

Then $\tsum_{i\in S}u_{i,\cdot }\notin V(S)-$Max$V(S)$

We conclude that $u=(u_{1\cdot },u_{2\cdot },...,u_{n\cdot })\in SO(N,V).$

\section{Exchange economies}

We consider a pure exchange economy $\mathcal{E}=(X_{i},e_{i},U_{i})_{i\in
N} $ with a finite number of agents, $N=\{1,2,...,n\}.$ The commodity space
is the Euclidean space $\mathbb{R}^{m}$. Each agent $i\in N$ is
characterized by her consumption set $X_{i}=\mathbb{R}^{m}$, her initial
endowment $e_{i}\in \mathbb{R}_{+}^{m}$ and her utility function $%
U_{i}:\tprod_{i\in N}X_{i}\rightarrow \mathbb{R}$. An allocation is an
element $x_{i}\in \mathbb{R}_{+}^{m}.$ An allocation $x$ is a feasible
allocation if $\tsum_{i\in N}x_{i}\leq \tsum_{i\in N}e_{i}.$

Let $p=mn.$ We will use the following notation: instead of $%
x=(x_{1},x_{2},...,x_{n})=(x_{1}^{1},x_{1}^{2},...,x_{1}^{m},x_{2}^{1},x_{2}^{2},...,x_{2}^{m},...,x_{n}^{1},x_{n}^{2},...,x_{n}^{m})\in 
\mathbb{R}_{+}^{p},$ we will use $x=(x_{1},x_{2},...,x_{m},$

\noindent $x_{1+m},...,x_{2m},...,x_{(n-1)m+1},...,x_{nm}),$ where $%
(x_{i}^{1},x_{i}^{2},...,x_{i}^{m})=(x_{(i-1)m+1},...,x_{im}).$

For each $S\subseteq N,$ we define the problem $P(S):$

max $U_{1}(x)=\frac{N_{1}(x)}{D_{1}(x)}=\frac{%
c_{11}x_{1}+c_{12}x_{2}+...+c_{m}x_{m}+c_{10}}{%
d_{11}x_{1}+d_{12}x_{2}+...+d_{1m}x_{m}+d_{10}}$

\ \ \ \ \ \ $U_{2}(x)=\frac{N_{2}(x)}{D_{2}(x)}=\frac{%
c_{21}x_{l+1}+c_{22}x_{l+2}+...+c_{2m}x_{2m}+c_{20}}{%
d_{21}x_{l+1}+d_{22}x_{l+2}+...+d_{2m}x_{m+2}+d_{20}}$

...

\ \ \ \ \ $U_{n}(x)=\frac{N_{n}(x)}{D_{n}(x)}=\frac{%
c_{n1}x_{(m-1)n+1}+c_{n2}x_{(m-1)n+2}+...+c_{nm}x_{mn}+c_{n0}}{%
d_{n1}x_{(m-1)n+1}+d_{n2}x_{(m-1)n+2}+...+d_{nm}x_{mn}+d_{n0}}$

$a_{11}(S)x_{1}+a_{12}(S)x_{2}+...+a_{1p}(S)x_{p}\leq b_{1}(S)$

$a_{21}(S)x_{1}+a_{22}(S)x_{2}+...+a_{2p}(S)x_{p}\leq b_{2}(S)$

................................................ \ \ \ \ \ \ \ \ \ \ \ \ \ \
\ \ \ \ \ \ \ \ \ \ \ \ \ \ \ \ (9)

$a_{m1}(S)x_{1}+a_{m2}(S)x_{2}+...+a_{mp}(S)x_{p}\leq b_{m}(S)$

$a_{11}(N)x_{1}+a_{12}(N)x_{2}+...+a_{1p}(N)x_{p}\leq b_{1}(N)$

$a_{21}(N)x_{1}+a_{22}(N)x_{2}+...+a_{2p}(N)x_{p}\leq b_{2}(N)$

................................................ \ \ \ \ \ \ \ \ \ \ \ \ \ \
\ \ \ \ \ \ \ \ \ \ \ \ 

$a_{m1}(N)x_{1}+a_{m2}(N)x_{2}+...+a_{mp}(N)x_{p}\leq b_{m}(N)$

$x_{1},x_{2},...,x_{p}\geq 0,\medskip $

where $%
c_{i}=(c_{i1},c_{i2},...,c_{im})^{T},d_{i}=(d_{i1},d_{i2},...,d_{im})^{T}\in 
\mathbb{R}^{m}$ for each $i\in \{1,2,...,n),$ $x\in \mathbb{R}^{p},$ $%
A=(a_{ij})_{\substack{ i=\overline{1,m}  \\ j=\overline{1,p}}}\in \mathbb{R}%
^{m\times p},$ $b(S)=(b_{1}(S),b_{2}(S),...,b_{m}(S))^{T}\in \mathbb{R}^{m},$
$D_{S}=\{x\in \mathbb{R}_{+}^{p}:A(S)x\leq b(S),$ $A(N)x\leq b(N),$ $x\geq
0\}\subseteq \mathbb{R}_{+}^{p}.\medskip $

The coefficients $a_{ij}(S)$ are defined as follows.

For $S\subset N,$ $a_{ij}(S)=\left\{ 
\begin{array}{c}
1\text{ if }j=(k-1)m+i,\text{ }i\in \{1,2,...,m-1\}\text{ and }k\in S; \\ 
0\text{ if }j\in (k-1)m+i,\text{ }i\in \{1,2,...,m-1\}\text{ and }k\notin S;
\\ 
1\text{ \ \ \ \ \ \ \ \ \ \ \ \ \ \ \ \ \ \ \ \ if \ \ \ \ \ \ \ \ \ \ \ \ \
\ \ \ \ \ \ }j=km\text{ and }k\in S; \\ 
0\text{ \ \ \ \ \ \ \ \ \ \ \ \ \ \ \ \ \ \ \ if \ \ \ \ \ \ \ \ \ \ \ \ \ \
\ \ \ \ \ \ }j=km\text{ and }k\notin S.%
\end{array}%
\right. $ for each $i\in \{1,2,...,m\}$ and $j\in \{1,2,...,p\}.\medskip $

For $S\subseteq N,$ the coefficients $b_{i}(S)=\tsum_{j\in S}e_{j}^{i}$
represents the initial endowment from the $i^{th}$ good of the coalition $S.$

The multiobjective linear fractional programming problem (9) is equivalent
with the multiobjective linear programming problem (10):\medskip

Max $\{tN_{i}(\frac{y}{t})$, $i\in N\}$

$tD_{i}(\frac{y}{t})\leq 1$ if $i\in N................(10)$

$A(S)(\frac{y}{t})-b(S)\leq 0$

$A(N)(\frac{y}{t})-b(N)\leq 0$

$t\in \mathbb{R}_{+},$ $y\in \mathbb{R}_{+}^{m},$ or, explicitely,$\medskip $

Max $c_{i}y+c_{i0}t,$ $i=1,...,m;$

$d_{i}y+d_{i0}t\leq 1,$ $i=1,...,m;$

$A(S)y-tIb(S)\leq 0....................(11)$

$A(N)y-tIb(N)\leq 0$

$t\in \mathbb{R}_{+},$ $y\in \mathbb{R}_{+}^{m}$.$\medskip $

We will attach the following cooperative game $(N,V)$ to the economy $%
\mathcal{E},$ $V:\mathcal{P}(N)\rightarrow \mathbb{R}^{n},$ $V(\emptyset
)=\{0\}$, $V(S)=($Max$\widehat{T}_{S}-\mathbb{R}_{+}^{n})\cap \mathbb{R}%
_{+}^{n}$ for each $S\subset N,$

$V(N)=($Max$\gamma \widehat{T}_{N}-\mathbb{R}_{+}^{n})\cap \mathbb{R}%
_{+}^{n} $

where $\gamma >\max (\gamma ^{\ast },n)$ and $\gamma ^{\ast }=\max_{\mathcal{%
B}}\tsum_{S\in \mathcal{B}}\gamma (S),$ $\mathcal{B}$ being any balanced
coalition of $N,$ $T_{S}=\{(t,y):(t,y)$ verifies the restrictions of the
problem $(3)\};$ $\widehat{T}_{S}=\{z\in \mathbb{R}^{m}:z=(tN_{i}(\frac{y}{t}%
))_{i\in \{1,2,...,m\}},$ $(t,y)\in T_{s}\}.$

We obtain the following results.

\begin{theorem}
The game $(N,V)$ is balanced.
\end{theorem}

\begin{proof}
Let $\mathcal{B}$ be a balanced colection of $N.$ First, we have that

$\tsum_{S\subset N}\gamma (S)b_{k}(S)=\tsum_{S\subset N}\tsum_{i\in S}\gamma
(S)b_{k}^{i}(S)=\tsum_{i\in N}\{\tsum_{S\subset N,S\backepsilon i}\gamma
(S)\}b_{k}^{i}=\tsum_{i\in N}b_{k}^{i}=b_{k}(N)$

for each $k\in \{1,2,...,m\}.$

Let $z(y(S),t(S))=(z_{1}(y(S),t(S)),z_{2}(y(S),t(S)),...,z_{l}(y(S),t(S)))%
\in V(S)$ and then,

$\tsum_{S\in \mathcal{B}}\frac{\gamma (S)}{\gamma ^{\prime }}z(y(S),t(S))=$

$=\tsum_{S\in \mathcal{B}}\frac{\gamma (S)}{\gamma ^{\prime }}%
(...,c_{i}y(S)+y(S)c_{i0},...)=$

$=(\widehat{t}N_{1}(\frac{\widehat{y}}{\widehat{t}}),\widehat{t}N_{2}(\frac{%
\widehat{y}}{\widehat{t}}),...,\widehat{t}N_{l}(\frac{\widehat{y}}{\widehat{t%
}}))=$

$=z(\widehat{y},\widehat{t}),$

where $\gamma ^{^{\prime }}=\tsum_{S\in \mathcal{B}}\gamma (S),$ $\widehat{y}%
:=\tsum_{S\in \mathcal{B}}\frac{\gamma (S)}{\gamma ^{\prime }}y(S)$ and $%
\widehat{t}=\tsum_{S\in \mathcal{B}}\frac{\gamma (S)}{\gamma ^{\prime }}%
t(S), $ $\widehat{y}\in \mathbb{R}_{+}^{p},\widehat{t}\in \mathbb{R}_{+}.$

Assume that $(y(S),t(S))\in D_{S}.$

We have that $A(S)y(S)\leq t(S)b(S)$ and $A(N)y(S)\leq t(S)b(N)$ for each $%
S\in \mathcal{B}$, $\frac{\gamma (S)}{\gamma ^{\prime }}\geq 0.$
Consequently, $\frac{\gamma (S)}{\gamma ^{\prime }}A(N)y(s)\leq t(S)\frac{%
\gamma (S)}{\gamma ^{\prime }}b(N)$. By adding, we obtain $\tsum_{S\in 
\mathcal{B}}A(N)\frac{\gamma (S)}{\gamma ^{\prime }}y(s))\leq \tsum_{S\in 
\mathcal{B}}\frac{\gamma (S)}{\gamma ^{\prime }}A(N)y(s)\leq \tsum_{S\in 
\mathcal{B}}\frac{\gamma (S)}{\gamma ^{\prime }}t(S)b(N)\leq b(N)\tsum_{S\in 
\mathcal{B}}t(S)\frac{\gamma (S)}{\gamma ^{\prime }},$ and then $%
A(N)(\tsum_{S\in \mathcal{B}}\frac{\gamma (S)}{\gamma ^{\prime }}y(s))\leq
b(N)(\tsum_{S\in \mathcal{B}}\frac{\gamma (S)}{\gamma ^{\prime }}t(S)).$
Therefore, $A(N)\widehat{y}\leq b(N)\widehat{t}$, that is $(\widehat{y},%
\widehat{t})$ verifies $A(N)y-tb(N)\leq 0$

Since $dy(S)+d_{0}t(S)\leq 1$ for each $S\in \mathcal{B}$ and $\frac{\gamma
(S)}{\gamma ^{\prime }}\geq 0,$ it follows that $d\tsum_{S\in \mathcal{B}}%
\frac{\gamma (S)}{\gamma ^{\prime }}y(S)+d_{0}\tsum_{S\in \mathcal{B}}\frac{%
\gamma (S)}{\gamma ^{\prime }}t(S)\leq 1$ and then, $d\widehat{y}+d_{0}%
\widehat{t}\leq 1.$

We notice that $\widehat{y},\widehat{t}\geq 0$ and conclude that $(\widehat{y%
},\widehat{t})\in V(N)$ and therefore $\tsum_{S\subset N}\gamma
(S)(y(S),t(S))\in V(N),$.....

...$\tsum_{S\subset N}\gamma (S)V(S)\subset V(N)$, that is the game is
balanced.
\end{proof}

We consider the dual problem to the multiobjective linear problem in order
to find a point beloging to the stable outcome $SO(N,V)=\{x\in \mathbb{R}%
_{+}^{m\times n}:x_{S}\notin V(S)\backslash $Max$V(S),$ $\forall S\subset
N\}.$

The next theorem gives an element of the Stable outcome.

\begin{theorem}
Let $\omega ^{\ast }$ be a Pareto optimal solution of the dual problem of
the associated multiobjective linear programming problem $(3)$ with $S=N.$
Then the payoff $u=(u_{1\cdot },u_{2\cdot },...,u_{n\cdot })\in \mathbb{R}%
^{n\times n},$ $u_{i\cdot }=(u_{i1},u_{i2},...,u_{in})$ defined by $u_{ik}=%
\frac{1}{n}\omega _{k1}^{\ast },$ $i=1,2,...,n$ and $k=1,2,...,n$ belongs to
the stable outcome of the game $(N,V).$
\end{theorem}

\begin{proof}
We will formulate first the multiobjective linear production programming
problems $P(S)$ for each $S\subseteq N$ equivalent to the multiobjective
fractional linear production programming problems and the dual.

Let $p=nm,$ $c^{\prime }\in \mathbb{R}^{n\times (p+1)},$ $A^{\prime }\in 
\mathbb{R}^{(n+2m)\times (p+1)}$ and $b^{\prime }\in \mathbb{R}^{n+2m}$ be
defined as $c^{\prime }=(c_{1},c_{10},...,c_{r},c_{r0}),$ $A^{\prime }(S)=%
\begin{pmatrix}
(d_{i})_{i\in I} & (d_{i0})_{i\in I} \\ 
A(S) & -I_{m}b(S) \\ 
A(N) & -I_{m}b(N)%
\end{pmatrix}%
b^{\prime }=(1_{I},0_{\mathbb{R}^{m}},0_{\mathbb{R}^{m}}),$ where $%
1_{I}=(1,1,...,1)\in \mathbb{R}^{m}$ and $0_{\mathbb{R}^{m}}=(0,0,...,0)\in 
\mathbb{R}^{m}.$
\end{proof}

\ \ \ \ \ \ \ \ \ \ \ \ \ \ \ \ \ \ \ \ \ \ For each $S\subsetneq N,$ the
problem $P(S)$ is\medskip

$P(S):$ \ \ \ \ \ \ \ Max $c^{\prime }(y,t)$

$\ \ \ \ \ \ \ \ \ \ \ \ \ \ \ \ (y,t)\in T_{p}=\{(y,t):A^{\prime
}(S)(y,t)=b^{\prime },$ $(y,t)\in \mathbb{R}_{+}^{p+1}\}\medskip $

and $P(N):$ \ \ \ \ \ \ \ Max $\gamma c^{\prime }(y,t)$

$\ \ \ \ \ \ \ \ \ \ \ \ \ \ \ \ (y,t)\in T_{p}=\{(y,t):A^{\prime
}(S)(y,t)=b^{\prime },$ $(y,t)\in \mathbb{R}_{+}^{p+1}\}\medskip $

Let $\widehat{T}_{S}$ , $\widehat{T}_{N}$ be the feasible areas in the
objective space of primal problems $P(S),$ resp. $P(N).$\medskip

\ \ \ \ \ \ \ \ \ \ \ \ \ \ \ \ \ \ \ \ \ For each $S\subsetneq N,$ the dual 
$D(S)$ is\medskip

$D(S):$ \ \ \ \ \ \ \ Min $\omega b^{\prime }$

$\ \ \ \ \ \ \ \ \ \ \ \ \ \ \ \ \omega \in T_{d}=\{\omega :\omega A^{\prime
}u\leq c^{\prime }u$ for no $u\in \mathbb{R}_{+}^{p+1}\}\medskip $

and $D(N)$ is Min $\omega b^{\prime }$

$\ \ \ \ \ \ \ \ \ \ \ \ \ \ \ \ \omega \in T_{d}=\{\omega :\omega A^{\prime
}u\leq \gamma c^{\prime }u$ for no $u\in \mathbb{R}_{+}^{p+1}\}\medskip $

Let $\omega ^{\ast }$ and $(y^{\ast },t^{\ast })$ be Pareto optimal
solutions for the problems $D(N)$ and $P(N)$.

It follows that $\omega ^{\ast }b^{\prime }=c^{\prime }(y^{\ast },t^{\ast })$
and then, $\omega ^{\ast }b^{\prime }\in $Max$\widehat{T}_{N}.$ We have that 
$\tsum_{i\in N}u_{i,\cdot }=\tsum_{i\in N}\frac{1}{n}\omega _{1,\cdot
}=(\omega _{11}^{\ast },...,\omega _{1n}^{\ast })\in $Max$V(N).$

For each $S\subsetneq N,$ $\tsum_{i\in S}u_{i,\cdot }=\tsum_{i\in N}\frac{|S|%
}{n}\omega _{1,\cdot }^{\ast }$

$\omega ^{\ast }A^{\prime }(N)u\leq \gamma c^{\prime }u$ for no $u\in 
\mathbb{R}_{+}^{p+1}$ implies that $\frac{\omega ^{\ast }}{\gamma }A^{\prime
}(S)u\leq c^{\prime }u$ for no $u\in \mathbb{R}_{+}^{p+1}.$ It follows that $%
\frac{\omega ^{\ast }b^{\prime }}{\gamma }\in V(S).$ Since $\gamma \geq n,$ $%
\frac{1}{\gamma }\leq \frac{1}{n}\leq \frac{|S|}{n}$ for each $S\subseteq N.$

$\tsum_{i\in S}u_{i,\cdot }=\tsum_{i\in N}\frac{|S|}{n}\omega _{1,\cdot
}^{\ast }\geq \frac{1}{\gamma }\omega _{1}^{\ast }.$

Then $\tsum_{i\in S}u_{i,\cdot }\notin V(S)-$Max$V(S)$

We conclude that $u=(u_{1\cdot },u_{2\cdot },...,u_{n\cdot })\in SO(N,V).$

\end{document}